\documentclass[12pt]{amsart}
\usepackage{amsfonts}
\usepackage{ifthen}
\usepackage{amsthm}
\usepackage{amsmath}
\usepackage{graphicx}
\usepackage{amscd,amssymb,amsthm}
\usepackage{graphicx}
\usepackage{epstopdf}
\usepackage{hyperref}
\usepackage{cite}
\usepackage[T1]{fontenc}
\usepackage[utf8]{inputenc}
\usepackage[dvipsnames]{xcolor}
\usepackage{cleveref}

\newcounter{minutes}
\divide\time by 60
\newcounter{hours}
\multiply\time by 60 \addtocounter{minutes}{-\time}

\newtheorem{lemma}{Lemma}
\newtheorem{theorem}{Theorem}

\newtheorem{definition}{Definition}

\newcommand{\real}{\operatorname{Re}}


\title[{Lemniscate and Exponential Starlikeness}]{Lemniscate and Exponential Starlikeness of Regular Coulomb Wave Functions}

\author[\.{I}. Akta\c{s}]{\.{I}brah\.{I}m Akta\c{s}}
\address{Department of Mathematics, Kam\.{ı}l \"{O}zda\u{g} Science Faculty, Karamano\u{g}lu Mehmetbey University, Yunus Emre Campus, 70100, Karaman--Turkey}
\email{aktasibrahim38@gmail.com; ibrahimaktas@kmu.edu.tr}

\begin{document}

\def\thefootnote{}
\footnotetext{ \texttt{File:~\jobname .tex,
          printed: \number\year-\number\month-\number\day,
          \thehours.\ifnum\theminutes<10{0}\fi\theminutes}
} \makeatletter\def\thefootnote{\@arabic\c@footnote}\makeatother

\maketitle

\begin{abstract}
	In this study, a normalized form of regular Coulomb wave function is considered. By using differential subordination method due to Miller and Mocanu, we determine some conditions on the parameters such that the normalized regular Coulomb wave function is lemniscate starlike and exponenetial starlike in the open unit disk, respectively.  
\end{abstract}
\section{Introduction and Preliminaries}
Let $ \mathbb{R} $, $\mathbb{C}$ and $ \mathbb{Z}^{+} $ denote the sets of real numbers, complex numbers and pozitive integers, respectively. By $\mathcal{H}[a,n]$, we denote the class of all analytic functions defined in the open unit disk $\mathbb{D}=\{z\in\mathbb{C}:\left|z\right|<1\}$ of the form $f(z)=a+a_{n}z^{n}+a_{n+1}z^{n+1}+\cdots,$ where $a\in\mathbb{C}$ and $n\in\mathbb{Z}^{+}.$ Let  $\mathcal{A}$ be the subclass of $ \mathcal{H}[0,1]$ consisting of functions $f$ normalized by the condition $f(0)=0\text{ and } f^{\prime}(0)=1$ and set $ \mathcal{H}[1,1]=\mathcal{H}_{1}$. By $\mathcal{S}$, we mean the class of functions belonging to $\mathcal{A}$, which are univalent in $\mathbb{D}$. For the analytic functions $f$ and $g$ in $\mathbb{D}$, the function $f$ is said to be subordinate to $g$, written $f\prec g$, if there exist a function $w$ analytic in $\mathbb{D}$, with $w(0)=0$ and $\left|w(z)\right|<1$, and such that $f(z)=g(w(z))$. If $g$ is univalent, then $f\prec g$ if and only if $f(0)=g(0)$ and $f(\mathbb{D})\subset g(\mathbb{D}).$

The class of starlike functions consists of all those functions $f\in\mathcal{A}$ such that the domain $f(\mathbb{D})$ is starlike with respect to origin. This function class is denoted by $\mathcal{S}^{\star}$ and has the following analytic characterization:
\begin{equation}\label{Starlike}
\mathcal{S}^{\star}=\bigg\{f\in\mathcal{A}:\real\left(\frac{zf^{\prime}(z)}{f(z)}\right)>0,z\in\mathbb{D}.\bigg\}
\end{equation}
In the literature, there is a function class
\begin{equation}\label{Ma-Minda starlike}
\mathcal{S}^{\star}(\varphi)=\bigg\{f\in\mathcal{S}:\frac{zf^{\prime}(z)}{f(z)}\prec\varphi(z)\bigg\}
\end{equation}
introduced by Ma and Minda in \cite{Ma-Minda} and is known as \textit{Ma-Minda starlike functions}. Here, the function $\varphi$ is analytic and univalent on $\mathbb{D}$ for which $\varphi(z)$ is starlike with respect to $\varphi(0)=1$ and it is symetric about the real axis with $\varphi^{\prime}(0)>0$. By the particular choosing of the function $\varphi$, many subclasses of starlike functions were defined in the literature. For example, if $\varphi(z)=(1+Az)/(1+Bz)$, where $-1\leq B<A\leq1$, the class $\mathcal{S}^{\star}[A,B]=\mathcal{S}^{\star}\left((1+Az)/(1+Bz)\right)$ is called the class of \textit{Janowski starlike functions} \cite{Janowski}. For $A=1-2\alpha$ and $B=-1$ with $\alpha\in\left[0,1\right)$, the class $\mathcal{S}^{\star}(\alpha)=\mathcal{S}^{\star}[1-2\alpha,-1]$ is known as \textit{starlike functions of order} $\alpha$ and was introduced by Robertson \cite{Robertson}. Setting $\alpha=0$, we obtain the class of starlike function given by \eqref{Starlike}. Also, by taking $\varphi=\sqrt{1+z}$, Sok\'{o}l and Stankiewicz \cite{Sokol-Stankiewicz} introduced the class $\mathcal{S}^{\star}(\sqrt{1+z})=\mathcal{SL}.$ In terms of subordination principle, the function $ f $ is called \textit{lemniscate starlike}  if $zf^{\prime}(z)/f(z)\prec\sqrt{1+z}.$ On the other hand, for $\varphi(z)=e^{z},$ Mendiratta \textit{et al.} \cite{Mendiratta} defined the class $\mathcal{S}_{e}^{\star}=\mathcal{S}^{\star}(e^{z})$ of starlike functions associated with the exponential function satisfying  the condition $\left|\log\left(\frac{zf^{\prime}(z)}{f(z)}\right)\right|<1.$ In additon to above, the authors in \cite{ACRK,Kanas,KKRC} gave some important results by using differential subordination method introduced by Miller and Mocanu \cite{Miller-Mocanu}. Comprehensive information about the differential subordination may be found in \cite{Miller-Mocanu} and \cite{Miller-Mocanu81}. Now, we would like to remind the basics of differential subordination principle.

Let $Q$ denote the set of analytic and univalent functions $q$ in $ \overline{\mathbb{D}}{\setminus}E(q) $, where $E(q)=\{\zeta\in\partial\mathbb{D}: \lim_{z\to\zeta}q(z)=\infty\}$ and such that $q^{\prime}(z)\neq0$ for $\zeta\in\partial\mathbb{D}\setminus{E(q)}.$ 
\begin{definition}\cite[p.27]{Miller-Mocanu}
	Let $\Omega$ be a set in $ \mathbb{C} $, $q\in Q$ and $n\in\mathbb{Z}^{+}.$ The class of admissible functions $\Psi_{n}[\Omega,q]$, consists of those functions $\Psi:\mathbb{C}^{3}\times\mathbb{D}\to\mathbb{C}$ that satisfy the admissibility condition:
	\begin{equation}
	\Psi(r,s,t;z)\notin\Omega,
	\end{equation}
	whenever $$ r=q(\zeta), s=m\zeta{q^{\prime}(\zeta)}, \real\left(1+\frac{t}{s}\right)\geq{m}\real\left[1+\frac{\zeta{q^{\prime\prime}(\zeta)}}{q^{\prime}(\zeta)}\right],$$where $z\in\mathbb{D}, \zeta\in\partial\mathbb{D}\setminus{E(q)}\text{ and }m\geq{n}.$ We write $\Psi_{1}[\Omega,q]=\Psi[\Omega,q].$ 
\end{definition} 
 \begin{theorem}\cite[Theorem 2.3b, p.28]{Miller-Mocanu}\label{Subordination The.}
	Let $\psi\in\Psi_{n}[\Omega,q]$ with $q(0)=a.$ If $p\in\mathcal{H}[a,n]$ satisfies
	$$\psi(p(z),zp^{\prime}(z),z^{2}p^{\prime\prime}(z))\in\Omega$$
	then $p(z)\prec{q(z)}.$
\end{theorem}
In \cite{Miller-Mocanu}, the authors discussed the class of admissible functions $\Psi[\Omega,q]$ when the function $q$ maps $ \mathbb{D} $ onto a disk or a half-plane. Very recently, taking the subordinate function $q(z)=\sqrt{1+z}$ Madaan \textit{et al.} \cite{MKR1} discussed the admissible function class $\Psi[\Omega,\sqrt{1+z}]$ and  gave a particular case of the Theorem\eqref{Subordination The.} as follow:
\begin{lemma}\cite{MKR1}\label{Lemma-Madaan}
	Let $p\in\mathcal{H}[1,n]$ with $p(z)\not\equiv1$ and $n\geq1.$ Let $\Omega\subset\mathbb{C}$ and $\Psi:\mathbb{C}^{3}\times\mathbb{D}\to\mathbb{C}$ satisfies the admissibility condition $ \Psi(r,s,t;z)\notin\Omega $ whenever $z\in\mathbb{D},$
	for $$r=\sqrt{2\cos{2\theta}}e^{i\theta}, s=\frac{me^{3i\theta}}{2\sqrt{2\cos{2\theta}}}\text{ and }\real\left(1+\frac{t}{s}\right)\geq\frac{3m}{4}$$ where $m\geq n\geq1$ and $-\frac{\pi}{4}<\theta<\frac{\pi}{4}$. If $$\Psi\left(p(z),zp^{\prime}(z),z^{2}p^{\prime\prime}(z);z\right)\in\Omega$$ for $z\in\mathbb{D},$ then $ p(z)\prec{\sqrt{1+z}}.$
\end{lemma}
Here, it is worth to mention that in case of first order diffrential subordination, the admissibility condition reduces to  $$\Psi\left(\sqrt{2\cos{2\theta}}e^{i\theta},\frac{me^{3i\theta}}{2\sqrt{2\cos{2\theta}}};z\right)\notin\Omega$$ where $z\in\mathbb{D}, \theta\in\left(-\frac{\pi}{4},\frac{\pi}{4}\right)$ and $m\geq n\geq1.$

In addition, Naz \textit{et al.} studied about the class of admissible function associated with the exponantial function $e^{z}$. In \cite{NNR1}, taking the subordinate function $q(z)=e^{z}$, the authors gave another special case of the Theorem\eqref{Subordination The.} as follow:
\begin{lemma}\cite{NNR1}\label{Lemma-Naz}
	Let $\Omega$ be a subset of $ \mathbb{C} $ and the function $\Psi:\mathbb{C}^{3}\times\mathbb{D}\to\mathbb{C}$ satisfies the admissibility condition $ \Psi(r,s,t;z)\notin\Omega $ whenever
	$$r=e^{e^{i\theta}}, s=me^{i\theta}r\text{ and }\real\left(1+\frac{t}{s}\right)\geq m(1+\cos\theta)$$where $z\in\mathbb{D}, \theta\in\left[0,2\pi\right)$ and $m\geq1.$ If $p$ is analytic function in $ \mathbb{D} $ with $p(0)=1$ and $$\Psi\left(p(z),zp^{\prime}(z),z^{2}p^{\prime\prime}(z);z\right)\in\Omega$$ for $z\in\mathbb{D},$ then $ p(z)\prec{e^{z}}.$
\end{lemma}
It is important to emphisaze here that the admissibility condition $ \Psi(r,s,t;z)\notin\Omega $ is verified for all $r=e^{e^{i\theta}}, s=me^{i\theta}e^{e^{i\theta}}$ and $ t $ with $ \real\left(1+\frac{t}{s}\right)\geq0 $, that is $\real\left((t+s)\right)e^{-i\theta}e^{e^{-i\theta}}\geq0$ for all $\theta\in\left[0,2\pi\right)$ and $m\geq1.$ Also, for the case  $\Psi:\mathbb{C}^{2}\times\mathbb{D}\to\mathbb{C}$, the admissibility condition reduces to $$\Psi\left(e^{e^{i\theta}},me^{i\theta}e^{e^{i\theta}};z\right)\notin\Omega$$where $z\in\mathbb{D}, \theta\in\left[0,2\pi\right)$ and $m\geq1.$

Since generalized hypergeometric functions has been used in the solution of famous Bieberbach conjecture, there has been a considerable interest on geometric properties of special functions. In the last few decades, many mathematicians started to investigate geometric properties (like univalence, starlikeness, convexity and close-to convexity) of some special functions including Bessel, Struve, Lommel, Wright and thier some generalizations. For these investigations the readers are referred to the papers \cite{aktas1,aktas2,aktas3,aktas4,aktas5,publ,BKS,samy,basz,basz2,BY,KKRC,MKR1,MKR2,Mendiratta,NNR1,NNR2,Ronning,TAH} and the references therein. Also, the authors studied some geometric properties of regular Coulomb wave function in \cite{Baricz,BCDT}, while the author investigated the zeros of regular Coulomb wave functions and their derivatives in \cite{Ikebe}. Motivated by the above works our main aim is to investigate the lemniscate and exponential starlikeness of regular Coulomb wave function by using differential subordination method.

This paper is organized as follow: The rest of this part is devoted definition of regular Coulomb wave function and its some properties. In Section\eqref{Main Results}, we deal with the lemniscate and exponential starlikeness of regular Coulomb wave functions.

The following second order differential equation:
\begin{equation}\label{Coulomb Dif. Eq.}
\frac{d^{2}w}{dz^{2}}+\left(1-\frac{2\eta}{z}-\frac{L(L+1)}{z^{2}}\right)w=0
\end{equation}
is known as Coulomb differential equation, see \cite{Abramowitz}. The equation \eqref{Coulomb Dif. Eq.} has two linearly independent solutions that are called regular and irregular Coulomb wave functions. The regular Coulomb wave function $F_{L,\eta}(z)$ is defined by (see \cite{BCDT})
\begin{equation}
F_{L,\eta}(z)=z^{L+1}e^{-iz}C_{L}(\eta){_{1}F_{1}\left(L+1-i\eta,2L+2;2iz\right)}=C_{L}(\eta)\sum_{n\geq0}a_{L,n}z^{n+L+1},
\end{equation}
where $z\in\mathbb{C}, L, \eta\in\mathbb{C}$,
$$C_{L}(\eta)=\frac{2^{L}e^{-\frac{\pi\eta}{2}}\left|\Gamma(L+1+i\eta)\right|}{\Gamma(2L+2)},$$
$$a_{L,0}=1,\hskip1cm a_{L,1}=\frac{\eta}{L+1},\hskip1cm a_{L,n}=\frac{2\eta{a_{L,n-1}-a_{L,n-2}}}{n(n+2L+1)},\hskip0.5cmn\in\{2,3,\dots\}$$ and ${_{1}F_{1}}$ denotes the Kummer confluent hypergeometric function. It is known from \cite{BCDT} and \cite{Stampach} that the regular Coulomb wave function $F_{L,\eta}(z)$ has the following Weierstrassian canonical product representation:
\begin{equation}\label{Weierstrassian canonical product}
F_{L,\eta}(z)=C_{L}(\eta)z^{L+1}e^{\frac{\eta{z}}{L+1}}\prod_{n\geq1}\left(1-\frac{z}{\rho_{L,\eta,n}}\right)e^{\frac{z}{\rho_{L,\eta,n}}},
\end{equation}
where $\rho_{L,\eta,n}$ denotes the $ n $th zero of the Colulomb wave function. 

In this study, since the regular Coulomb wave function $F_{L,\eta}(z)$ does not belong to the class $ \mathcal{A} $ we consider the following normalized form:
\begin{equation}\label{normalized form}
g_{L,\eta}(z)=C_{L}^{-1}(\eta)z^{-L}F_{L,\eta}(z).
\end{equation}

\section{Main Results}\label{Main Results}
\setcounter{equation}{0}
In this section we determine some conditions on the parameters such that  the regular Coulomb wave function is lemniscate and exponential starlike in the unit disk $ \mathbb{D} $. Our first main result is the following and it is related to the lemniscate starlikeness of normalized regular Coulomb wave function $z\mapsto{g_{L,\eta}(z)}.$
\begin{theorem}\label{Main Theorem1}
	Let $\eta, L\in\mathbb{C}.$ If $$(\sqrt{2}-1)\left|2L-1\right|+2\left|\eta\right|<\frac{\sqrt{2}}{4},$$then the normalized regular Coulomb wave function $z\mapsto{g_{L,\eta}(z)}$ is lemniscate starlike in the unit disk $ \mathbb{D}. $
\end{theorem}
\begin{proof}
	In order to prove our assertion, we shall use Lemma\eqref{Lemma-Madaan}. Now, define the function $P_{L,\eta}:\mathbb{D}\to\mathbb{C}$ by
	\begin{equation}\label{P function}
	P_{L,\eta}(z)=\frac{zg_{L,\eta}^{\prime}(z)}{g_{L,\eta}(z)}.
	\end{equation}
	It is clear from the equality \eqref{Weierstrassian canonical product} that the normalized regular Coulomb wave function $g_{L,\eta}(z)$ can be represented by the next product:
	\begin{equation}\label{product of normalized form}
	g_{L,\eta}(z)=ze^{\frac{\eta{z}}{L+1}}\prod_{n\geq1}\left(1-\frac{z}{\rho_{L,\eta,n}}\right)e^{\frac{z}{\rho_{L,\eta,n}}}.
	\end{equation}
	Taking logarithmic derivation of \eqref{product of normalized form}  and by multiplying by $z$ obtained equality we can write that
	\begin{equation*}
	P_{L,\eta}(z)=\frac{zg_{L,\eta}^{\prime}(z)}{g_{L,\eta}(z)}=1+\frac{\eta}{L+1}z+\sum_{n\geq1}\frac{z^{2}}{\rho_{L,\eta,n}(z-\rho_{L,\eta,n})}.
	\end{equation*}
	As a result, the function $P_{L,\eta}$ is analytic in $ \mathbb{D} $ and $P_{L,\eta}(0)=1$. On the other hand, since regular Coulomb wave function $F_{L,\eta}(z)$ satisfies Coulomb differential equation given by \eqref{Coulomb Dif. Eq.}, it is easily seen that the function $g_{L,\eta}(z)$ satisfies the following differential equation:
	\begin{equation}\label{Dif eq. of g}
	z^{2}g_{L,\eta}^{\prime\prime}(z)+
	2Lzg_{L,\eta}^{\prime}(z)+
	(z^{2}-2\eta{z}-2L)g_{L,\eta}(z)=0.
	\end{equation}
	Taking logarithmic derivation  of the function $P_{L,\eta}(z)$ given by \eqref{P function} yields that
	\begin{equation}\label{logarithmic derivative of P}
	\frac{zg_{L,\eta}^{\prime\prime}(z)}{g_{L,\eta}^{\prime}(z)}=\frac{zP_{L,\eta}^{\prime}(z)-P_{L,\eta}^{2}(z)-P_{L,\eta}(z)}{P_{L,\eta}(z)}.
	\end{equation}
	Now, if we consider equations \eqref{P function} and \eqref{logarithmic derivative of P} in \eqref{Dif eq. of g}, then the function $P_{L,\eta}(z)$ satisfies the next differential equation:
	\begin{equation}\label{Dif. eq. of P}
	zP_{L,\eta}^{\prime}(z)+P_{L,\eta}^{2}(z)+(2L-1)P_{L,\eta}(z)+z^{2}-2\eta{z}-2L=0.
	\end{equation}
	Define the function $\Psi:\mathbb{C}^{2}\times\mathbb{D}\to\mathbb{D},$
	\begin{equation}\label{Psi function}
	\Psi(r,s;z)=s+r^{2}+(2L-1)r+z^{2}-2\eta{z}-2L
	\end{equation}
	and set $\Omega=\{0\}.$ It is seen from \eqref{Psi function} that 
	$$\Psi(P_{L,\eta}(z),zP_{L,\eta}^{\prime}(z);z)=0\in\Omega$$
	for $\forall z\in\mathbb{D}.$ In wiev of Lemma\eqref{Lemma-Madaan}, we have to show that $\Psi(r,s;z)\notin\Omega$ for $$r=\sqrt{2\cos{2\theta}}e^{i\theta}, s=\frac{me^{3i\theta}}{2\sqrt{2\cos{2\theta}}},-\frac{\pi}{4}<\theta<\frac{\pi}{4}, m\geq n\geq1\text{ and }\forall z\in\mathbb{D}.$$
	By using reverse triangle inequality in \eqref{Psi function}, we can write
	\begin{equation}\label{Ineq. of Psi}
	\left|\Psi(r,s;z)\right|\geq\left|s+r^{2}-1\right|-\left|2L-1\right|\left|r-1\right|-\left|z\right|^{2}-2\left|\eta\right|\left|z\right|.
	\end{equation}
	We would like to find minimum of $\left|s+r^{2}-1\right|$ and maximum of $\left|r-1\right|$. For this purpose consider
	\begin{equation*}
	s+r^{2}-1=\frac{me^{3i\theta}}{2\sqrt{2\cos{2\theta}}}+2\cos{2\theta}e^{2i\theta}-1=e^{3i\theta}\left(\frac{m}{2\sqrt{2\cos{2\theta}}}+e^{i\theta}\right).
	\end{equation*}
	Therefore, we obtain
	\begin{equation*}
	\left|s+r^{2}-1\right|^{2}=\frac{m^{2}}{8\cos{2\theta}}+\frac{m\cos\theta}{\sqrt{2\cos{2\theta}}}+1.
	\end{equation*}
	Define the function $U:\left(-\frac{\pi}{4},\frac{\pi}{4}\right)\to\mathbb{R}$, $$U(\theta)=\frac{m^{2}}{8\cos{2\theta}}+\frac{m\cos\theta}{\sqrt{2\cos{2\theta}}}+1.$$
	Since
	\begin{equation*}
	U^{\prime}(\theta)=\frac{m^{2}\sin{2\theta}}{4\cos^{2}{2\theta}}+\frac{m\sin\theta}{2\sqrt{2}\cos{2\theta}\sqrt{\cos{2\theta}}}
	\end{equation*}
	the function $U(\theta)$ has a critical point at the point $\theta=0$ in $\left(-\frac{\pi}{4},\frac{\pi}{4}\right).$ Also, from the second derivative test there is a minimum in $\theta=0$ since 
	\begin{equation*}
	U^{\prime\prime}(0)=\frac{2m^{2}+\sqrt{2}m}{4}>0.
	\end{equation*}
	Namely,
	\begin{equation*}
	\min_{\theta\in(-\frac{\pi}{4},\frac{\pi}{4})}U(\theta)=U(0)=\frac{(m+2\sqrt{2})^{2}}{8}.
	\end{equation*}
	As a result, we obtain
	\begin{equation*}
	\left|s+r^{2}-1\right|^{2}\geq\frac{(m+2\sqrt{2})^{2}}{8}
	\end{equation*}
	or
	\begin{equation}\label{Minimum1}
	\left|s+r^{2}-1\right|\geq\frac{1+2\sqrt{2}}{2\sqrt{2}}
	\end{equation}
	for $m\geq1.$ On the other hand, we can write that
	\begin{equation*}
	\left|r-1\right|^{2}=\left|\sqrt{2\cos{2\theta}}e^{i\theta}-1\right|^{2}=2\cos{2\theta}-2\sqrt{2}\cos{\theta}\sqrt{\cos{2\theta}}+1.
	\end{equation*}
	If we define the function $V:\left(-\frac{\pi}{4},\frac{\pi}{4}\right)\to\mathbb{R}$,
	\begin{equation*}
	V(\theta)=2\cos{2\theta}-2\sqrt{2}\cos{\theta}\sqrt{\cos{2\theta}}+1,
	\end{equation*}
	then we see that
	\begin{equation*}
	V^{\prime}(\theta)=\sin{\theta}\left(\frac{2\sqrt{2}}{\sqrt{\cos{2\theta}}}-8\cos{\theta}\right).
	\end{equation*}
	So, the critical point of the function $V(\theta)$ is $\theta=0$ in $(-\frac{\pi}{4},\frac{\pi}{4})$. Also, there is a maximum at  the point  $\theta=0$ since $V^{\prime\prime}(0)=-8+2\sqrt{2}<0.$ That is,
	\begin{equation*}
	\max_{\theta\in(-\frac{\pi}{4},\frac{\pi}{4})}V(\theta)=V(0)=(\sqrt{2}-1)^{2}.
	\end{equation*}
	Therefore, we have that
	\begin{equation*}
	\left|r-1\right|^{2}\leq(\sqrt{2}-1)^{2}
	\end{equation*}
	or
	\begin{equation}\label{Maximum1}
	\left|r-1\right|\leq(\sqrt{2}-1).
	\end{equation}
	Finally, if we consider the inequalities \eqref{Minimum1} and \eqref{Maximum1} in the inequality\eqref{Ineq. of Psi}, then we get
	\begin{equation*}
	\left|\Psi(r,s;z)\right|\geq\frac{\sqrt{2}}{4}-(\sqrt{2}-1)\left|2L-1\right|-2\left|\eta\right|.
	\end{equation*}
	This shows that $\left|\Psi(r,s;z)\right|>0$ under the hypothesis. So, $\Psi(r,s;z)\neq0$  and by virtue of the Lemma\eqref{Lemma-Madaan} the function $z\mapsto{g_{L,\eta}(z)}$ is lemniscate starlike in the unit disk $ \mathbb{D}$.
\end{proof}

 The following is the second main result and it is related to the exponential starlikeness of normalized regular Coulomb wave function $z\mapsto{g_{L,\eta}(z)}.$

\begin{theorem}\label{Main Theorem2}
	Let $\eta, L\in\mathbb{C}.$ If $$(e-1)\left|2L-1\right|+2\left|\eta\right|<\frac{e-1}{e^{2}},$$then the normalized regular Coulomb wave function $z\mapsto{g_{L,\eta}(z)}$ is exponential starlike in the unit disk $ \mathbb{D}. $
\end{theorem}
\begin{proof}
	In order to prove the exponential starlikeness of the normalized regular Coulomb wave function $z\mapsto{g_{L,\eta}(z)}$ we shall use the Lemma\eqref{Lemma-Naz} due to Naz \textit{et al.} It is known from the Theorem\eqref{Main Theorem1} that the function $P_{L,\eta}(z)$ is analytic in $\mathbb{D}$ and $P_{L,\eta}(0)=1.$ Now, consider again the function $\Psi(r,s;z)$ given by \eqref{Psi function} and suppose that $\Omega=\{0\}.$ It is clear that $$\Psi(P_{L,\eta}(z),zP_{L,\eta}^{\prime}(z);z)=0\in\Omega$$
	for $\forall z\in\mathbb{D}.$ By virtue of Lemma\eqref{Lemma-Naz}, we need to show that $\Psi(r,s;z)\notin\Omega$ whenever $$r=e^{e^{i\theta}}, s=me^{i\theta}e^{e^{i\theta}}, \real\left((t+s)\right)e^{-i\theta}e^{e^{-i\theta}}\geq0, \theta\in\left[0,2\pi\right),  z\in\mathbb{D} \text{ and }m\geq1.$$ Here, we want to calculate the minimum value of $\left|s+r^{2}-1\right|$ and the maximum value of $\left|r-1\right|$ under our assumption. If we consider
	\begin{align*}
	s+r^{2}-1&=me^{i\theta}e^{e^{i\theta}}+e^{e^{2i\theta}}-1\\&=\left[me^{\cos\theta}\cos(\theta+\sin\theta)+e^{2\cos\theta}\cos(2\sin\theta)-1\right]\\&+i\left[me^{\cos\theta}\sin(\theta+\sin\theta)+e^{2\cos\theta}\sin(2\sin\theta)\right],
	\end{align*}
	then we have
	\begin{align*}
	\left|s+r^{2}-1\right|^{2}&=\left[me^{\cos\theta}\cos(\theta+\sin\theta)+e^{2\cos\theta}\cos(2\sin\theta)-1\right]^{2}\\&+\left[me^{\cos\theta}\sin(\theta+\sin\theta)+e^{2\cos\theta}\sin(2\sin\theta)\right]^{2}.
	\end{align*}
	By using second derivative test, it can be shown that the function $ A:\left[0,2\pi\right)\to\mathbb{R} $,
	 \begin{align*}
	 A(\theta)&=\left[me^{\cos\theta}\cos(\theta+\sin\theta)+e^{2\cos\theta}\cos(2\sin\theta)-1\right]^{2}\\&+\left[me^{\cos\theta}\sin(\theta+\sin\theta)+e^{2\cos\theta}\sin(2\sin\theta)\right]^{2}
	 \end{align*} has a minimum at the point $\theta=\pi$ in $\left[0,2\pi\right).$ That is,
	 \begin{equation*}
	 \min_{\theta\in\left[0,2\pi\right)}A(\theta)=A(\pi)=\left(\frac{1}{e^{2}}-\frac{m}{e}-1\right)^{2}
	 \end{equation*}
	and so,
	\begin{equation}\label{Minimum2}
		\left|s+r^{2}-1\right|\geq\frac{1}{e}-\frac{1}{e^{2}}-1
	\end{equation}
	for $m\geq1.$ On the other hand, we have that 
	\begin{equation*}
	\left|r-1\right|^{2}=\left|e^{e^{i\theta}}-1\right|^{2}=e^{2\cos\theta}-2e^{\cos\theta}\cos(\sin\theta)+1.
	\end{equation*}
	Define the function $B:\left[0,2\pi\right)\to\mathbb{R}$,
	$$B(\theta)=e^{2\cos\theta}-2e^{\cos\theta}\cos(\sin\theta)+1.$$ 
    It can be easily shown that the function $B(\theta)$ attains its maximum at the point $\theta=0$ in $\left[0,2\pi\right).$ Namely,
    \begin{equation*}
    \max_{\theta\in\left[0,2\pi\right)}B(\theta)=B(0)=e^{2}-2e+1=(e-1)^{2}.
    \end{equation*}
    As a consequence, we get
    \begin{equation}\label{Maximum2}
    \left|r-1\right|\leq{e-1}.
    \end{equation}
    Taking into acount the inequalities \eqref{Minimum2} and \eqref{Maximum2} in the inequality \eqref{Ineq. of Psi} we can write that
    \begin{equation*}
    \left|\Psi(r,s;z)\right|\geq\frac{e-1}{e^{2}}-(e-1)\left|2L-1\right|-2\left|\eta\right|>0.
    \end{equation*}
    This means that $\Psi(e^{e^{i\theta}},me^{i\theta}e^{e^{i\theta}};z)\notin\Omega.$ In view of Lemma\eqref{Lemma-Naz} the normalized regular Coulomb wave function $z\mapsto{g_{L,\eta}(z)}$ is exponential starlike in the unit disk $ \mathbb{D}. $ 
\end{proof}


\begin{thebibliography}{99}
	
	\bibitem{Abramowitz}
	\textsc{M. Abramowitz, I.A. Stegun},
	\textit{Handbook of Mathematical Functions with Formulas, Graphs, and Mathematical Tables}, Dover Publications, NewYork, 1972.
	
	\bibitem{aktas1}
	\textsc{\.{I}. Akta\c{s}, \'A. Baricz}, Bounds for radii of starlikeness of some $q$-Bessel functions, {\em Results Math.} 72(1) (2017) 947--963.
	
	
	\bibitem{aktas2}
	\textsc{\.{I}. Akta\c{s}, \'A. Baricz, H. Orhan}, Bounds for the radii of starlikeness and convexity of some special functions, {\em Turkish J. Math.} 42(1) (2018) 211--226.
	
	\bibitem{aktas3}
	\textsc{\.{I}. Akta\c{s}, \'A. Baricz and S. Singh}, Geometric and monotonic properties of hyper-Bessel functions, {\em Ramanujan J}, (2019), https://doi.org/10.1007/s11139-018-0105-9.
	
	\bibitem{aktas4}
	\textsc{\.{I}. Akta\c{s}, \'A. Baricz, N. Ya\u{g}mur}, Bounds for the radii of univalence of some special functions, {\em Math. Inequal. Appl.} 20(3) (2017) 825--843.
	
	
	\bibitem{aktas5}
	\textsc{\.{I}. Akta\c{s}, H. Orhan}, Bounds for the radii of convexity of some $q$-Bessel functions, Bulletin of the Korean Mathematical Society, (in press).
	
	\bibitem{ACRK}
	\textsc{R.M. Ali, N.E. Cho, V. Ravichandran, S.S. Kumar}, \textit{Differential subordination for
	functions associated with the lemniscate of Bernoulli}, Taiwanese J. Math., \textbf{16}(3) (2012),
	1017--1026.
	
	\bibitem{publ} \textsc{\'A. Baricz}, Geometric properties of generalized Bessel functions, {\em Publ. Math. Debrecen} 73 (2008) 155--178.
	
	\bibitem{Baricz}
	\textsc{\'{A}. Baricz}, \textit{Tur\'{a}n type inequalities for regular Coulomb wave functions}, J. Math. Anal. Appl. 430 (2015) 166–180.
	
	\bibitem{BCDT}
	\textsc{\'{A}. Baricz, M. \c{C}a\u{g}lar, E. Deniz, E. Toklu}, \textit{Radii of starlikeness and convexity of regular Coulomb wave functions}, arXiv:1605.06763
	
	\bibitem{BKS} \textsc{\'A. Baricz, P.A. Kup\'an, R. Sz\'asz}, The radius of starlikeness of normalized Bessel functions of the first kind, {\em Proc. Amer. Math. Soc.} 142(6) (2014) 2019--2025. 
	
	
	
	\bibitem{samy} \textsc{\'A. Baricz, S. Ponnusamy}, Starlikeness and convexity of generalized Bessel functions, {\em Integr. Transforms Spec. Funct.} 21 (2010) 641--653.
	
	\bibitem{basz} \textsc{\'A. Baricz, R. Sz\'asz}, The radius of convexity of normalized Bessel functions of the first kind, {\em Anal. Appl.} 12(5) (2014) 485--509.
	
	\bibitem{basz2} \textsc{\'A. Baricz, R. Sz\'asz}, Close-to-convexity of some special functions, {\em Bull. Malay. Math. Sci. Soc.} 39(1) (2016) 427--437.
	
	\bibitem{BY} \textsc{\'{A}. Baricz, N. Ya\u{g}mur}, Geometric properties of some Lommel and Struve functions, {\em Ramanujan J.} (in press) doi:10.1007/s11139-015-9724-6.
	
	\bibitem{Ikebe}
	\textsc{Y. Ikebe}
	\textit{The Zeros of Regular Coulomb Wave Functions and of Their Derivative}, Mathematics of Computation, \textbf{29}(131) (1975), 878--887.
	
	\bibitem{Janowski}
	\textsc{W. Janowski}
	\textit{Extremal problems for a family of functions with positive real part and for some related families}.
	Annales Polonici Mathematici, \textbf{23} (1970), 159-177.
	
	\bibitem{Kanas}
	\textsc{S. Kanas}, \textit{Differential subordination related to conic sections}, J. Math. Anal. Appl. \textbf{317}(2)
	(2006), 650--658.
	
	\bibitem{KKRC}
	\textsc{S.S. Kumar, V. Kumar, V. Ravichandran, N.E. Cho}, \textit{Sufficient conditions for starlike functions associated with the lemniscate of Bernoulli}, J. Inequal. Appl. 2013, 2013:176, 13
	pp.
	
	\bibitem{Ma-Minda}
	\textsc{W. Ma, C.D. Minda}
	\textit{A unified treatment of some special classes of univalent functions}, Proceedings of the International Conference on Complex Analysis at the Nankai Institute of Mathematics, (1992) 157--169.
	
	
	\bibitem{MKR1}
	\textsc{V. Madaan, A. Kumar, V. Ravichandran}, \textit{Starlikeness Associated with Lemniscate of Bernoulli}, Filomat, (in press).
	

	\bibitem{MKR2}
	\textsc{V. Madaan, A. Kumar, V. Ravichandran}, \textit{Lemniscate convexity and other properties of generalized Bessel functions}, arXiv:1902.04277v1
	
	\bibitem{Mendiratta}
	\textsc{R. Mendiratta, S. Nagpal, V. Ravichandran},
	\textit{On a subclass of strongly starlike functions associated with exponantiel function},
	Bulletin of Malaysian Mathematical Sciences Society, \textbf{38}(1) (2015), 365--386.
	
	
	\bibitem{Miller-Mocanu}
	\textsc{S.S. Miller, P.T. Mocanu}, \textit{Differential Subordinations. Theory and Applications}, Marcel Dekker, Inc., New York-Basel, 2000.
	 
	\bibitem{Miller-Mocanu81}
	\textsc{S.S. Miller, P.T. Mocanu}, \textit{Differential subordinations and univalent functions}, The Michigan Mathematical Journal,\textbf{28}(2) (1981), 157--172.
	
	\bibitem{NNR1}
	\textsc{A. Naz, S. Nagpal, V. Ravichandran}, \textit{Starlikeness associated with the exponential function,} Turkish Journal of Mathematics, Volume 43 (2019), 1353--1371.
	
	
	\bibitem{NNR2}
	\textsc{A. Naz, S. Nagpal, V. Ravichandran}, \textit{ Exponential starlikeness and convexity of confluent hypergeometric, Lommel and Struve functions}, arXiv:1908.072v1
	
	\bibitem{Robertson}
	\textsc{M. S. Robertson}, \textit{Certain classes of starlike functions}, Michigan Math. J. \textbf{32}(2) (1985),
	135--140.
	
	\bibitem{Ronning}
	\textsc{F. R\o{}nning}
\textit{Uniformly convex functions and a corresponding class of star-like functions }. Proceedings of the American	Mathematical Society, 118 (1) (1993), 189--196.
	
	\bibitem{Sokol-Stankiewicz}
	\textsc{J. Sok\'{o}\l{}, J. Stankiewicz}, \textit{Radius of convexity of some subclasses of strongly starlike functions}, Zeszyty Nauk. Politech Rzeszowskiej Mat.,19 (1996), 101--105.
	
	
	
	\bibitem{Stampach}
	\textsc{F. \v{S}tampach, P. \v{S}\v{T}ov\'{\i}\v{c}ek},
	\textit{Orthogonal polynomials associated with Coulomb wave functions}, J. Math. Anal. Appl.
	419(1) (2014) 231–25
	
	\bibitem{TAH}
	\textsc{E. Toklu, \.{I}. Akta\c{s}, H. Orhan}
	\textit{Radii problems for normalized $q$-Bessel and Wright functions}, Acta Univ. Sapientiae, Mathematica, \textbf{11}(1) (2019), 203--223.
	\end{thebibliography}
\end{document}